\documentclass[10 pt]{amsart}

\usepackage{amsmath} 
\usepackage{amsfonts}
\usepackage{amssymb}
\usepackage{amstext}
\usepackage{amsbsy}
\usepackage{amsopn}
\usepackage{amsthm}
\usepackage{amsxtra}
\usepackage{graphicx}
\usepackage{caption}
\usepackage{subcaption}
\usepackage{color}
\usepackage{hyperref}
\usepackage{enumerate}
\usepackage{amscd}

\newtheorem{theorem}{Theorem}[section]
\newtheorem{lemma}[theorem]{Lemma}

\newtheorem*{conjecture*}{Conjecture}
\newtheorem*{claim*}{Claim}
\newtheorem*{theorem*}{Theorem}

\theoremstyle{remark}
\newtheorem{remark}[theorem]{Remark}
\theoremstyle{definition}
\newtheorem{definition}[theorem]{Definition}

\newcommand{\A}{\mathcal{A}}

\newcommand{\Z}{\mathbb{Z}}

\newcommand{\N}{\mathbb{N}}
\newcommand{\rst}[1]{\ensuremath{{\mathbin\upharpoonright}%
\raise-.5ex\hbox{$#1$}}}

\newcommand{\Aut}{{\rm Aut}}

\title{The automorphism group of a shift of subquadratic growth}
\author{Van Cyr}
\address{Bucknell University, Lewisburg, PA 17837 USA}
\email{van.cyr@bucknell.edu}
\author{Bryna Kra}
\address{Northwestern University, Evanston, IL 60208 USA}
\email{kra@math.northwestern.edu}

\subjclass[2010]{37B50 (primary), 68R15, 37B10}
\keywords{subshift, automorphism, block complexity}

\thanks{The  second author was partially supported by NSF grant $1200971$.}

\begin{document}

\begin{abstract}
For a subshift over a finite alphabet,  a measure of the complexity of the system is obtained 
by counting the number of nonempty cylinder sets of length $n$.  When this complexity grows exponentially, 
the automorphism group has been shown to be large for various classes of subshifts.  In contrast, we show 
that subquadratic growth of the complexity implies that for a topologically transitive shift $X$, the automorphism group 
$\Aut(X)$ is small: 
if $H$ is the subgroup of $\Aut(X)$ generated by the shift, then $\Aut(X)/H$ is periodic.  

\end{abstract}

\maketitle

\section{Introduction}
In this note, we study the group of automorphisms of a subshift.  More precisely, if $\A$ is a finite alphabet with the discrete topology and we endow $\A^{\Z}$ with the product topology, a closed set $X\subseteq\A^{\Z}$ is called a {\em subshift} it is invariant under the {\em left shift map} $\sigma\colon\A^{\Z}\to\A^{\Z}$ that acts on $x\in\A^{\Z}$ by $(\sigma x)(i+1):=x(i)$ for all $i\in\Z$.  An automorphism of $(X,\sigma)$ is a homeomorphism $\varphi\colon X\to X$ that commutes with $\sigma$.  We denote the group of automorphisms of $(X, \sigma)$ by $\Aut(X)$.

There are numerous theorems showing that the automorphism group can be extremely large for different classes of subshifts.
The first such result was proven by Curtis, Hedlund and Lyndon (see Hedlund~\cite{Hedlund2}), who 
showed that $\Aut(\A^{\Z})$ contains isomorphic copies of any finite group and also 
contains two involutions whose product has infinite order.  
For mixing one dimensional subshifts of finite type (of which $\A^{\Z}$ is an example), Boyle, Lind and Rudolph~\cite{BLR} showed that the automorphism group contains the 
free group on two generators, the direct sum of countably many copies of $\Z$, and the direct sum of every countable collection of finite groups.  
These theorems show that $\Aut(X)$ is large, and are proven by constructing automorphisms that generate subgroups with prescribed properties.  In contrast, we are interested in placing restrictions on $\Aut(X)$, showing that for certain classes of subshifts, the automorphism group can not contain certain structures, and so we need a different approach.

There are cases in which the automorphism group can be characterized.  For example, 
Host and Parreau~\cite{HP} gave a complete description for primitive substitutions of constant 
length and this was generalized in Salo and T\"orm\"a~\cite{SaTo}.  
Olli~\cite{Olli} described the automorphism group of Sturmian shifts, and generalizations are 
given in~\cite{DDMP}.

For each result showing that the automorphism group is large, there is a notion of complexity associated with the system and this complexity is large.  More precisely, for a shift system, let $P_X(n)$ denote the number of nonempty cylinder sets of length $n$.  For the full shift and for mixing subshifts of finite type, this complexity grows exponentially and this growth is an important ingredient in the constructions.  
We study the opposite situation, where the complexity has slow growth and we show that this places  strong restrictions on $\Aut(X)$.  In particular, we are interested in shifts $(X,\sigma)$ for which $P_X(n)$ grows subquadratically (see Section~\ref{sec:complexity} for the precise definition of the growth).  To state our main theorem, recall that a (possibly infinite) group is {\em periodic} if every element has finite order.  We show:

\begin{theorem}
\label{subquad-thm}
Suppose $(X,\sigma)$ is a topologically transitive shift of subquadratic growth and let $H$ be the subgroup of $\Aut(X)$ generated by $\sigma$.  Then $\Aut(X)/H$ is a periodic group.
\end{theorem}

The collection of shifts of subquadratic growth includes many examples that arise naturally in symbolic dynamics and in the combinatorics of words. 
The theorem applies to Sturmian sfhits, and more generally to Arnoux-Rauzy shifts~\cite{AR} (which have linear complexity) 
and linearly recurrent systems~\cite{DHS}.
A theorem of Pansiot~\cite{Pansiot} shows that for a purely morphic shift $X$,  $P_X$  is one of $\Theta(1)$, $\Theta(n)$, $\Theta(n\log\log n)$, $\Theta(n\log n)$, or $\Theta(n^2)$, where $\Theta$ is the asymptotic growth rate, and all but the last class have subquadratic growth.  For more extensive literature on shift systems, see for example~\cite{PF}.

We conclude with a brief comment on the ideas in the proof of Theorem~\ref{subquad-thm}.   
Instead of working in the one dimensional setting, we use the one dimensional automorphisms to produce colorings of $\Z^2$.   By the complexity assumption on $\Aut(X)$, we can apply a theorem of Quas and Zamboni showing that these colorings are simple.  We then use this information on the two dimensional colorings to deduce the one dimensional result.  

\section{Complexity}
\label{sec:complexity}

\subsection{One dimensional shifts}
Throughout we assume that  $\A$ is a finite set.  For $x\in\A^{\Z}$, we write $x = (x(i)\colon i\in\Z)$ and let 
$x(i)$ denote the element of $\A$ that $x$ assigns to $i\in\Z$ .  
The shift map $\sigma\colon\A^{\Z}\to\A^{\Z}$ is defined by $(\sigma x)(i):=x(i+1)$.  With respect to the metric
$$
d(x,y):=2^{-\min\{|i|\colon x(i)\neq y(i)\}},
$$
$\A^{\Z}$ is compact and $\sigma$ is a homeomorphism. 

If $F\subset\Z$ is a finite set and $\beta\in \A^F$, then the {\em cylinder set} $[F;\beta]$ is defined as 
$$
[F;\beta] := \{x\in \A^\Z\colon x(i) = \beta(i) \text{ for all } i\in F\}.
$$
The collection of all cylinder sets is a basis for the topology of $\A^{\Z}$.

A closed, $\sigma$-invariant set $X\subseteq\A^{\Z}$ is called a {\em subshift}.  The group of all homeomorphisms from $X$ to itself that commute with $\sigma$ is called the {\em automorphism group of $X$} and is denoted $\Aut(X)$.  A classical result of Hedlund~\cite{Hedlund2} says that if $\varphi\in\Aut(X)$, then 
$\varphi$ is a {\em sliding block code}, meaning that 
there exists $N_{\varphi}\in\N$ such that for all $x\in X$ and all $i\in\Z$, $(\varphi x)(i)$ is determined entirely by $(x_{i-N_{\varphi}}, x_{i-N_{\varphi}+1},\dots,x_{i+N_{\varphi}-1},x_{i+N_{\varphi}})\in\A^{2N_{\varphi}+1}$.  An automorphism $\varphi$ has {\em range $N$} if $N_{\varphi}$ can be chosen to be $N$.

As a measure of the complexity of a given subshift $X$, the {\em block complexity function} $P_X\colon\N\to\N$ is defined by
$$
P_X(n)=\bigl\vert\{\beta\in\A^{B_n}\colon [\beta, B_n] \cap X\neq\emptyset\}\bigr\vert, 
$$
where $B_n :=\{x\in\Z\colon 0\leq x< n\}$.  
Defining the complexity $P_x(n)$ to be the number of configurations in a window of size $n$ in some fixed $x\in X$, we have that 
$$
P_X(n) \geq \sup_{x\in X} P_x(n), 
$$
and equality holds when the subshift $X$ is transitive.  

It is well-known that $P_X(n)$ is sub-multiplicative and so the {\em topological entropy} $h_{top}(X)$ of $X$ defined by
$$
h_{top}(X):=\lim_{n\to\infty}\frac{\log(P_X(n))}{n}, 
$$
is well-defined (see, for example~\cite{LM}).  For subshifts whose topological entropy is zero, one can study the {\em upper polynomial growth rate} of $X$ defined by
$$
\overline{P}(X):=\limsup_{n\to\infty}\frac{\log(P_X(n))}{\log(n)}\in[0,\infty]
$$
and the {\em lower polynomial growth rate} of $X$ given by
$$
\underline{P}(X):=\liminf_{n\to\infty}\frac{\log(P_X(n))}{\log(n)}\in[0,\infty].
$$

The classical Morse-Hedlund Theorem~\cite{MH} states that $x\in X$ is periodic if and only there exists some $n\in\N$ such that $P_x(n)\leq n$.  It follows immediately that if $X$ contains at least one aperiodic element (that is, at least one $x\in X$ for which $\sigma^ix\neq\sigma^jx$ for any $i\neq j$), then $\underline{P}(X)\geq1$.

\subsection{Two dimensional shifts}
With minor modifications, these notions extend to higher dimensions.  We only need the results in two dimensions and 
so only state the generalizations in this setting. 

If $\eta\in\A^{\Z^2}$, let $\eta(i,j)$ denote the entry $\eta$ assigns to $(i,j)\in\Z^2$.  With respect to the metric
$$
d(\eta_1,\eta_2)=2^{-\min\{\|\vec v\|\colon \eta_1(\vec v)\neq\eta_2(\vec v)\}}, 
$$
the space $\A^{\Z^2}$ is compact.  If $F\subset\Z^2$ is finite and $\beta\colon F\to\A$, then the {\em cylinder set}
$[F;\beta]$ is defined as 
$$
[F;\beta]:=\{\eta\in\A^{\Z^2}\colon\eta(i,j)=\beta(i,j)\text{ for all }(i,j)\in F\}.
$$
As for $\A^{\Z}$, the cylinder sets form a basis for the topology of $\A^{\Z^2}$.  We define the {\em left-shift} $S\colon\A^{\Z^2}\to\A^{\Z^2}$ by
$$
(S\eta)(i,j):=x(i+1,j)
$$
and the {\em down-shift} $T\colon\A^{\Z^2}\to\A^{\Z^2}$ by
$$
(T\eta)(i,j):=x(i,j+1).
$$
These maps commute and both are homeomorphisms of $\A^{\Z^2}$.

For $\eta\in\A^{\Z^2}$, we denote the {\em $\Z^2$-orbit of $\eta$} by
$$
\mathcal{O}(\eta):=\{S^aT^b\eta\colon(a,b)\in\Z^2\}  
$$
and let $\overline{\mathcal{O}}(\eta)$ denote the closure of $\mathcal{O}(x)$ in $\A^{\Z^2}$ (note that the $\Z^2$ action by the shifts $S$ and $T$ is implicit in this notation).  A closed subset $Y\subseteq\A^{\Z^2}$ is a {\em subshift} of $\A^{\Z^2}$ if it is both $S$-invariant and $T$-invariant.  In particular, for any fixed $\eta\in\A^{\Z^2}$ the set $\overline{\mathcal{O}}(\eta)$ is a subshift.

\subsection{Automorphisms and $\Z^2$ configurations}

Suppose $\varphi\in\Aut(X)$ and $x\in X$.  We define an element of $\A^{\Z^2}$ by:
\begin{equation}
\label{eq:phi-x}
\eta_{\varphi,x}(i,j):=(\varphi^jx)_i.
\end{equation}
For fixed $\varphi\in\Aut(X)$, the {\em inclusion map}  $\imath_{\varphi}\colon X\to\A^{\Z^2}$ given by $\imath_{\varphi}(x):=\eta_{\varphi,x}$ is a homeomorphism from $X$ to $\imath_{\varphi}(X)$ and satisfies $\imath_{\varphi}\circ\sigma=S\circ \imath_{\varphi}$ and $\imath_{\varphi}\circ\varphi=T\circ \imath_{\varphi}$.  That is, $\imath_{\varphi}$ is a topological conjugacy between the $\Z^2$-dynamical system $(X,\sigma,\varphi)$ and the $\Z^2$-dynamical system $(\imath_{\varphi}(X),S,T)$.  (Note that the joint action of $\sigma$ and $\varphi$ on $X$ is a $\Z^2$-dynamical system, as $\sigma$ commutes with $\varphi$.)

For a subshift $Y\subseteq\A^{\Z^2}$, the {\em rectangular complexity function} $P_Y\colon\N\times\N\to\N$ is defined by
$$
P_Y(n,k):=\left\vert\{\beta\in\A^{R_{n,k}}\colon [\beta; R_{n,k}]\cap Y\neq\emptyset\}\right\vert
$$
where $R_{n,k}:=\{(x,y)\in\Z^2\colon0\leq x<n\text{ and }0\leq y<k\}$.  (Again, we could have defined this for elements $y\in Y$ and taken a supremum and equality of these complexities holds for transitive systems.)

For $x\in X$, the crucial relationship between the (one-dimensional) block complexity $P_X$ and the (two-dimensional) rectangular complexity $P_{\overline{\mathcal{O}}(\imath_{\varphi}(x))}$ is the following:

\begin{lemma}\label{complexity-relation}
Suppose $X$ is a shift and $\varphi,\varphi^{-1}\in\Aut(X)$ are block codes of range $N$.  Then for any $x\in X$, we have 
\begin{equation}\label{eq:comp}
P_{\overline{\mathcal{O}}(\eta_{\varphi,x})}(n,k)\leq P_X(2Nk-2N+n).
\end{equation}
\end{lemma}
\begin{proof}
Suppose $[\beta;R_{n,k}]\cap \imath_{\varphi}(X)\neq\emptyset$.  Since $\imath_{\varphi}$ is a topological conjugacy between $(X,\sigma,\varphi)$ and $(\imath_{\varphi}(X),S,T)$, there exists $x\in X$ such that the restriction of $\eta_{\varphi,x}$ to $R_{n,k}$ is $\beta$.  Since $\varphi$ is a block code of range $N$, the word 
$$
(x_{-N-i_1},x_{-N-i_1+1},\dots,x_{N+i_2-1},x_{N+i_2})
$$
determines $(\varphi x)_j$ for all $-i_1\leq j\leq i_2$.  It follows inductively that the word
$$
(x_{-tN-i_1},x_{-tN-i_1+1},\dots,x_{tN+i_2-1},x_{tN+i_2})
$$
determines $(\varphi^rx)_j$ for $1\leq r\leq t$ and $(t-r)N-i_1\leq j\leq(t-r)N+i_2$.  In particular, the word
$$
(x_{-(k-1)N},x_{-(k-1)N+1},\dots,x_{(k-1)N+n-1},x_{(k-1)N+n-1})
$$
determines $(\varphi^rx)_j$ for all $0\leq r<k$ and all $0\leq j<n$.

This means that the restriction of $\eta_{\varphi,x}$ to the set $R_{n,k}$ is determined by the restriction of $\eta_{\varphi,x}$ to the set $\{(x,0)\in\Z^2\colon-(k-1)N\leq x\leq(k-1)N+n-1\}$.  By definition of $\eta_{\varphi,x}$, this is determined by the word
$$
(x_{-(k-1)N},x_{-(k-1)N+1},\dots,x_{(k-1)N+n-1},x_{(k-1)N+n-1}).
$$
The number of distinct colorings of this form is $P_X(2(k-1)N+n)$ and so the number of distinct $\beta\colon R_{n,k}\to\A$ for which $[\beta;R_{n,k}]\neq\emptyset$ is at most $P_X(2Nk-2N+n)$.
\end{proof}

It follows from Lemma~\ref{complexity-relation} that if
\begin{equation}\label{eq:subquadratic}
\liminf_{n\to\infty}\frac{P_X(n)}{n^2}=0,
\end{equation}
then
\begin{equation}
\liminf_{n\to\infty}\frac{P_{\overline{\mathcal{O}}(\imath_{\varphi}(x))}(n,n)}{n^2}=0.
\end{equation}

\begin{definition}\label{def:subquad}
We say that a shift $X\subseteq\A^{\Z}$ satisfying~\eqref{eq:subquadratic} is a shift of {\em subquadratic growth}.
\end{definition}

\begin{remark}
We remark that the condition that the shift X has subquadratic growth is related to a statement about the lower polynomial growth rate of $X$.  If $\underline{P}(X)<2$, then $X$ has subquadratic growth.  On the other hand, if $X$ has subquadratic growth, then $\underline{P}(X)\leq2$.  If $\underline{P}(X)=2$, then $X$ may or may not have subquadratic growth. 
\end{remark}

\section{Shifts of subquadratic growth}\label{sec:quad}

If $\varphi\in\Aut(X)$, then as $x\in X$ varies, our main tool to study the $\Z^2$-configurations that arise as $\eta_{\varphi,x}$ (as defined in~\eqref{eq:phi-x}) is the following theorem of Quas and Zamboni:

\begin{theorem}[Quas and Zamboni~\cite{QZ}]\label{Quas-Zamboni-thm}
Suppose $\eta\in\A^{\Z^2}$ and there exists $n,k\in\N$ such that $P_{\eta}(n,k)\leq nk/16$.  Then there is a finite set $F\subset\Z^2\setminus\{(0,0)\}$ (that depends only on $n$ and $k$) and a vector $\vec v\in F$ such that $\eta(\vec x+\vec v)=\eta(\vec x)$ for all $\vec x\in\Z^2$.
\end{theorem}

Although this is not the way their theorem is stated, by checking through the cases in the proof, this is 
exactly what they show in Theorem $4$ in~\cite{QZ}.

We use this to show: 
\begin{lemma}\label{subquadratic}
Suppose $X$ is a shift of subquadratic growth and $\varphi\in\Aut(X)$.  Then there exists a finite set $F\subset\Z^2\setminus\{(0,0)\}$ (which depends only on $\varphi$ and $X$) such that for all $x\in X$, there exists $(a_{(\varphi,x)}, b_{(\varphi,x)})\in F$ such that $S^{a_{(\varphi,x)}}T^{b_{(\varphi,x)}}\imath_{\varphi}(x)=\imath_{\varphi}(x)$.
\end{lemma}
\begin{proof}
Suppose
$$
\liminf_{n\to\infty}\frac{P_X(n)}{n^2}=0.
$$
Let $N$ be the range of the block code $\varphi$.  Find the smallest $n_1\in\N$ for which
$$
P_X(2N(n_1-1)+n_1)\leq n_1^2/16.
$$
By Theorem~\ref{Quas-Zamboni-thm}, there exists a finite set $F\subset\Z^2\setminus\{(0,0)\}$ (which depends only on $n_1$ and hence only on the subshift $X$) such that if $\eta\in\A^{\Z^2}$ satisfies $P_{\eta}(n_1,n_1)\leq n_1^2/16$, then there exists $\vec v\in F$ for which $\eta(\vec x+\vec v)=\eta(\vec x)$ for all $\vec x\in\Z^2$.

Let $x\in X$ be fixed.  By~\eqref{eq:comp},
$$
P_{\overline{\mathcal{O}}(\imath_{\varphi}(x))}(n_1,n_1)\leq n_1^2/16.
$$
Thus for some $(a_{\varphi,x},b_{\varphi,x})\in F$, we have that $S^{a_{\varphi,x}}T^{b_{\varphi,x}}\imath_{\varphi}(x)=\imath_{\varphi}(x)$.
\end{proof}

\begin{lemma}\label{lemma:subquad}
Suppose $X$ is a topologically transitive shift of subquadratic growth and let $\varphi\in\Aut(X)$.  Then there exists a vector $(a_{\varphi},b_{\varphi})\in\Z^2\setminus\{(0,0)\}$ such that for all $x\in X$, $S^{a_{\varphi}}T^{b_{\varphi}}\imath_{\varphi}(x)=\imath_{\varphi}(x)$.
\end{lemma}
\begin{proof}
By Lemma~\ref{subquadratic} there exists a finite set $F\subseteq\Z^2\setminus\{(0,0)\}$ such that for all $x\in X$ there exists $\vec v\in F$ such that $\imath_{\varphi}(x)$ is periodic with period vector $\vec v$.  For each $x\in X$ let
$$
V_x:=\{\vec v\in F\colon\vec v\text{ is a period vector of }\imath_{\varphi}(x)\}.
$$
Since $V_x\subseteq F$ and $F$ is finite, there exists $M\in\N$ such that whenever $V_x$ contains two linearly independent vectors (so that $\imath_{\varphi}(x)$ is doubly periodic) the vertical period of $\imath_{\varphi}(x)$ is at most $M$.  Therefore, if $V_x$ contains two linearly independent vectors for all $x\in X$, then $\imath_{\varphi}(x)$ is vertically periodic with period vector $(0,M!)$ for all $x\in X$.  In this case,  the vector $(a_{\varphi},b_{\varphi})=(0,M!)$ satisfies the conclusion of the lemma.

We are left with showing that if there exists $x\in X$ such that all of the vectors in $V_x$ are collinear, then there exists $(a_{\varphi},b_{\varphi})\in F$ such that $\imath_{\varphi}(x)$ is periodic with period $(a_{\varphi},b_{\varphi})$ for all $x\in X$.  Let 
$$
B:=\{x\in X\colon\dim({\rm Span}(V_x))=1\}
$$
be the set of ``bad points'' in $X$.  For each $x\in B$ let $v(x)$ be a shortest nonzero integer vector that spans ${\rm Span}(V_x)$ (there are two possible choices).  Fix some $x_0\in B$ and let $\vec v=v(x_0)$.  There are two cases to consider:

\medskip

\noindent {\em Case 1}: Suppose that $v(x)$ is collinear with $v(y)$ for any $x,y\in B$.  Fix $x_0\in B$ and let $\vec v\in\Z^2\setminus\{(0,0)\}$ be a shortest integer vector parallel to $v(x_0)$ (there are two possible choices).  Then for all $x\in B$, there exists $n_x\in\Z$ such that $v(x)=n_x\cdot v_{x_0}$.  Since $v(x)\in F$ and $F$ is finite, $\{n_x\colon x\in B\}$ is bounded.  For all $y\in X\setminus B$ the coloring $\imath_{\varphi}(y)$ is doubly periodic.  For each such $y$, choose $n_y\in\Z$ such that $n_y\cdot\vec v$ is a shortest nonzero period vector for $\imath_{\varphi}(y)$ parallel to $\vec v$.  Since $\imath_{\varphi}(y)$ has two linearly independent period vectors in $F$, the set $\{n_y\colon y\in X\setminus B\}$ is bounded.  Therefore if $N$ is the least common multiple of $\{|n_z|\colon z\in X\}$, then $N\cdot\vec v$ is a period vector for $\imath_{\varphi}(z)$ for all $z\in X$.  In this case, set $(a_{\varphi},b_{\varphi}):=N\cdot\vec v$.

\medskip

\noindent {\em Case 2}: Suppose there exist $x_1, x_2\in B$ such that $v(x_1)$ is not collinear with $v(x_2)$.  We  obtain a contradiction in this case, thereby completing the proof of the lemma.  Since $\dim({\rm Span}(V_{x_1}))=1$, for any $\vec w\in F\setminus V_{x_1}$ there exists $\vec y_{\vec w}\in\Z^2$ such that 
\begin{equation}\label{eq1}
\eta_{\varphi,x_1}(\vec y_{\vec w})\neq\eta_{\varphi,x_1}(\vec y_{\vec w}+\vec w).
\end{equation}
Choose $N_1\in\N$ such that the restriction of $\eta_{\varphi,x_1}$ to the set $\{(x,0)\colon-N_1\leq x\leq N_1\}$ determines $\eta_{\varphi,x_1}(\vec y_{\vec w})$ and $\eta_{\varphi,x_1}(\vec y_{\vec w}+\vec w)$ for all $\vec w\in F\setminus V_{x_1}$.  Similarly, for $\vec w\in F\setminus V_{x_2}$, there exists $\vec z_{\vec w}\in\Z^2$ such that
\begin{equation}\label{eq2}
\eta_{\varphi,x_2}(\vec z_{\vec w})\neq\eta_{\varphi,x_2}(\vec z_{\vec w}+\vec w).
\end{equation}
Choose $N_2\in\N$ such that the restriction of $\eta_{\varphi,x_2}$ to the set $\{(x,0)\colon-N_2\leq x\leq N_2\}$ determines $\eta_{\varphi,x_2}(\vec z_{\vec w})$ and $\eta_{\varphi,x_2}(\vec z_{\vec w}+\vec w)$ for all $\vec w\in F\setminus V_{x_2}$.

By topological transitivity of $(X,\sigma)$, there exists $\xi\in X$ and $a,b\in\Z$ such that
	\begin{eqnarray*}
	\xi(i-a)&=& x_1(i)\text{ for all }-N_1\leq i\leq N_1; \\
	\xi(i-b)&=& x_2(i)\text{ for all }-N_2\leq i\leq N_2.
	\end{eqnarray*}
Therefore for any $\vec w\in F\setminus V_{x_1}$ we have
	\begin{eqnarray}
	\eta_{\varphi,\xi}(\vec y_{\vec w}-(a,0))&=&\eta_{\varphi,x_1}(\vec y_{\vec w});\label{eq:s1} \\
	\eta_{\varphi,\xi}(\vec y_{\vec w}+\vec w-(a,0))&=&\eta_{\varphi,x_1}(\vec y_{\vec w}+\vec w),\label{eq:s2}
	\end{eqnarray}
and for any $\vec w\in F\setminus V_{x_2}$ we have
	\begin{eqnarray}
	\eta_{\varphi,\xi}(\vec y_{\vec w}-(b,0))&=&\eta_{\varphi,x_1}(\vec y_{\vec w});\label{eq:s3} \\
	\eta_{\varphi,\xi}(\vec y_{\vec w}+\vec w-(b,0))&=&\eta_{\varphi,x_1}(\vec y_{\vec w}+\vec w).\label{eq:s4}
	\end{eqnarray}
By Lemma~\ref{subquadratic}, $\eta_{\varphi,\xi}$ is periodic and its period vector lies in $F$.  Combining equations~\eqref{eq1},~\eqref{eq:s1}, and~\eqref{eq:s2} we see this vector is not an element of the set $F\setminus V_{x_1}$.  Similarly, by combining equations~\eqref{eq2},~\eqref{eq:s3}, and~\eqref{eq:s4}, we see this vector is not in the set $F\setminus V_{x_2}$.  Since $V_{x_1}\cap V_{x_2}=\emptyset$,  we obtain the desired contradiction.
\end{proof}

We use this lemma to complete the proof of Theorem~\ref{subquad-thm}:
\begin{proof}[Proof of Theorem~\ref{subquad-thm}]
Suppose $X$ is a shift of subquadratic growth and let $\varphi\in\Aut(X)$.  By Lemma~\ref{lemma:subquad}, there exists $(a_{\varphi},b_{\varphi})\in\mathbb{Z}^2\setminus\{(0,0)\}$ such that $S^{a_{\varphi}}T^{b_{\varphi}}\imath_{\varphi}(x)=\imath_{\varphi}(x)$ for all $x\in X$.  Since $\imath_{\varphi}$ is a topological conjugacy between $(X,\sigma,\varphi)$ and $(\imath_{\varphi}(X),S,T)$, we have that $\sigma^{a_{\varphi}}\varphi^{b_{\varphi}}x=x$ for all $x\in X$ and so $\varphi^{b_{\varphi}}=\sigma^{-a_{\varphi}}$.  Thus if $H$ is the subgroup of $\Aut(X)$ generated by the powers of $\sigma$, the projection of $\varphi^{b_{\varphi}}$ to $\Aut(X)/H$ is the identity.

Since this argument can be applied to any $\varphi\in\Aut(X)$ (where the parameters $a_{\varphi}$ and $b_{\varphi}$ depend on $\varphi$), it follows that $\Aut(X)/H$ is a periodic group.  

\end{proof}

\end{document}